\documentclass[12pt,a4paper]{article}%
\usepackage[utf8]{inputenc}
\usepackage{hyperref}
\usepackage{amsmath}
\usepackage{amsfonts}
\usepackage{amssymb}
\usepackage{xcolor}
\usepackage{graphicx}%
\providecommand{\U}[1]{\protect\rule{.1in}{.1in}}
\newtheorem{theorem}{Theorem}

\newtheorem{corollary}[theorem]{Corollary}

\newtheorem{definition}[theorem]{Definition}

\newtheorem{lemma}[theorem]{Lemma}

\newtheorem{proposition}[theorem]{Proposition}

\newenvironment{proof}[1][Proof]{\noindent\textbf{#1.} }{\ \hfill \rule{0.5em}{0.5em}\bigskip}
\graphicspath{{Slike/}{D:/Dropbox/Riste-Sedlar/20260716 Sigma Irregularity/Slike/}}

\begin{document}

\title{A complete characterization of maximally $\sigma$-irregular trees with
prescribed maximum degree}
\author{Martin Knor$^{1}$, Jelena Sedlar$^{2,3}$, Riste \v{S}krekovski$^{3,4}$\\{\small $^{1}$ \textit{Slovak University of Technology in Bratislava,
Slovakia}}\\[0.1cm] {\small $^{2}$ \textit{University of Split, Faculty of civil
engineering, architecture and geodesy, Croatia}}\\[0.1cm] {\small $^{3}$ \textit{Faculty of Information Studies, Novo mesto,
Slovenia}}\\[0.1cm] {\small $^{4}$ \textit{University of Ljubljana, Faculty of Mathematics
and Physics, Slovenia}}}
\date{}
\maketitle

\begin{abstract}
The $\sigma$-irregularity of a graph $G=(V,E)$ is defined as $\sigma
(G)=\sum_{uv\in E}(d(u)-d(v))^{2},$ where $d(u)$ denotes the degree of a
vertex $u.$
A tree on $n$ vertices with the maximum degree $\Delta$ is called
\emph{maximal} if it attains the greatest possible value of $\sigma
$-irregularity among all such trees.
The maximal trees are already known for
chemical trees, i.e. for $\Delta\leq4,$ and for $\Delta=5$.
In this paper we
characterize the maximal trees for every $\Delta\geq6$ when $n\geq
\Delta(\Delta-1)+1$.
We introduce three families of trees $\mathcal{T}_{n,\Delta
}^{\prime}$, $\mathcal{T}_{n,\Delta}^{\prime\prime}$ and $\mathcal{T}%
_{n,\Delta}^{\prime\prime\prime}$.
All the trees of the same family have the same $\sigma$-irregularity, and we
give an explicit formula for each of the three values.
Comparing the formulas tells which of the three values is the greatest for
given $n$ and $\Delta$.
We then show that a tree is maximal if and only if it belongs to a family
attaining that greatest value.
In contrast with the case $\Delta\leq5,$ the two natural candidate families
$\mathcal{T}_{n,\Delta}^{\prime}$ and $\mathcal{T}_{n,\Delta}^{\prime\prime}$
do not suffice, since for every $\Delta\geq7$ and every $n\equiv3\pmod{\Delta}$
the trees of $\mathcal{T}_{n,\Delta}^{\prime\prime\prime}$ have strictly
greater $\sigma$-irregularity.
\end{abstract}



\section{Introduction}

The graphs under consideration are simple and finite. Let $V(G)$ and $E(G)$
represent the vertex and edge sets of a graph $G$, respectively.
For a vertex $v\in V(G)$, we denote its degree in $G$ as $d_{G}(v)$.
We may omit the subscript $G$ when the context is clear.

A graph $G$ is \emph{regular} if all its vertices have the same
degree, and otherwise it is \emph{irregular}.
An \emph{irregularity measure} or \emph{irregularity index} $\mathrm{I}(G)$
is a topological invariant such that $\mathrm{I}(G) \geq0$, and
$\mathrm{I}(G) = 0$ holds if and only if $G$ is regular.
Various irregularity measures play an important role in many scientific areas
including chemistry and network theory \cite{cfar-cnlgacmi-2014, e-qnh-2010,
e-riicbn-2010, rsdh-giumdqsrps-2018, s-dvigh-1981}. Among the best-known and
most thoroughly investigated irregularity measures is the \emph{Albertson
irregularity index} \cite{AlbertsonIrr} defined as
\[
\mathrm{irr}(G)=\sum_{uv\in E(G)}\left\vert d_{G}(u)-d_{G}(v)\right\vert ,
\]
see \cite{Dimit-Abdo1, ad-iggo-2014, Bell1, GHM05, HansenMelot}.

In order to avoid the absolute value in the Albertson irregularity
index, one naturally arrives to the \emph{$\sigma$-irregularity index}
\[
{\sigma}(G)=\sum_{uv\in E(G)}(d_{G}(u)-d_{G}(v))^{2}.
\]
The first results on $\sigma$-irregularity index were obtained by Gutman et al.
in \cite{gtycc-ipsi}.
Graphs with maximal $\sigma$-irregularity were
characterized in \cite{adg-gmsi-2018}, where also some lower bounds on the
$\sigma$-irregularity were presented.
Inverse problem for the $\sigma$-irregularity was resolved in \cite{adg-gmsi-2018, gtycc-ipsi}.
Further, R\'{e}ti \cite{r-spgiiprsi-2019} compared $\sigma$-irregularity
with several other irregularity measures.
Connected $k$-cyclic graphs with maximal $\sigma$-irregularity were determined in
\cite{aaabh-msikcgi-2023}, and trees with maximal $\sigma$-irregularity under
a prescribed degree sequence were studied in \cite{dglc-etfdssi-2023}.

The problem of maximizing the $\sigma$-irregularity becomes particularly
interesting for trees with prescribed maximum degree $\Delta$. Chemical trees,
i.e. trees with $\Delta\leq4$, with maximal $\sigma$-irregularity were
characterized in \cite{kpvsd-sict-2024}.
This result was later extended to trees with
the maximum degree $\Delta=5$ in \cite{ksvsd-sipmt-2026}.
Very recently, the maximal trees with $\Delta=6$ were characterized for every
residue class of $n$ modulo $6$ \cite{b-tmsi-2026}.
For general $\Delta \geq4$, the exact maximum value of $\sigma$-irregularity was determined by a linear programming approach for the orders $n\equiv0$ and $n\equiv
1\pmod{\Delta}$ \cite{b-msit-2026}.

This paper is concerned with the \emph{general} case $\Delta\geq6$, which
remained open after \cite{ksvsd-sipmt-2026}. Throughout, a tree $T$ with
maximum degree $\Delta$ is called \emph{maximal} if it attains the largest
possible value of $\sigma(T)$ among all trees on $\left\vert V(T)\right\vert $
vertices with maximum degree $\Delta$.
For every $\Delta\geq6$ and every sufficiently large $n$, more precisely for
every $n=k\Delta+1+j$ with $k\geq\Delta-1$ and $0\leq j\leq\Delta-1$, we
establish the exact value of the maximum $\sigma$-irregularity and we
characterize all the trees attaining it. Namely, we introduce three explicitly
described families of trees, each of them having a constant value of the
$\sigma$-irregularity, and we prove that a tree on $n$ vertices with the
maximum degree $\Delta$ is maximal if and only if it belongs to a family among
those three whose value of the $\sigma$-irregularity is the greatest.

The paper is
organized as follows.
In Section~2 we introduce three families of trees
$\mathcal{T}_{n,\Delta}^{\prime}$, $\mathcal{T}_{n,\Delta}^{\prime\prime}$ and
$\mathcal{T}_{n,\Delta}^{\prime\prime\prime}$, and we state the main result of
the paper.
We also establish the exact maximum value of $\sigma$-irregularity.
Section~3 is devoted to the proof of the main theorem.
We first recall the structural properties of maximal trees from \cite{ksvsd-sipmt-2026} which hold for every $\Delta\geq3$, and we establish several further structural properties holding for $\Delta\geq 6$.
Finally, we combine these structural results to prove the main theorem.
We conclude the paper with some remarks in Section~4.


\section{Extremal families of trees}

In a tree, a vertex of degree at least $2$ is \emph{dominant
end-vertex} if it has precisely one neighbor with degree greater than $1$.
Following \cite{ksvsd-sipmt-2026}, sometimes we call a dominant end-vertex simply
an \emph{internal leaf}.
Let $T$ be a tree and  $v\in V(T)$.
If $d(v)=k$, then $u$ is a $k$-vertex.
Let $n_{i}$ denote the number of vertices in $T$ of degree $i$, where $i$
ranges from $1$ to $\Delta$.
Similarly, let $m_{ij}$, for $1\leq i\leq
j\leq\Delta$, be the number of edges in $T$ with end-vertices of
degrees $i$ and $j$.

An edge $uv$ of a tree $T$ is \emph{optimal} (resp. \emph{suboptimal}) if $(d(u)-d(v))^{2}$ equals $(\Delta-1)^{2}$ (resp. $(\Delta-2)^{2}$).
Observe that the
contribution of an optimal edge to $\sigma(T)$ is largest possible and the
contribution of a suboptimal edge is the second largest. An edge is optimal if
the degrees of its two end-vertices are $\Delta$ and $1$. On the other hand,
an edge $uv$ can be suboptimal in two ways:

\begin{itemize}
\item[$(1)$] $d(u)=\Delta$ and $d(v)=2$, in which case it is a \emph{good}
suboptimal edge;

\item[$(2)$] $d(u)=\Delta-1$ and $d(v)=1$, in which case it is a \emph{bad}
suboptimal edge.
\end{itemize}

\noindent A tree is \emph{nice} if it contains only optimal and suboptimal edges.
Number $n$ is \emph{nice} if $n=k\Delta+1$ for some integer $k\geq1$.

\begin{lemma}
\label{Lemma_badSuboptimal}
A nice tree with maximum degree $\Delta\geq4$
on $n\geq\Delta+1$ vertices does not contain bad suboptimal edges.
\end{lemma}

\begin{proof}
Let $T$ be a nice tree.
Assume to the contrary, that $T$ contains a bad
suboptimal edge.
Then $T$ must have a vertex, say $u$, of degree $\Delta-1$.
Since $n\geq\Delta+1$, at least one neighbor of $u$ is not a leaf, denote it by $v$. Since $2\leq d(v)\leq\Delta$
and $\Delta\geq4$ imply $\left\vert d(u)-d(v)\right\vert \leq\Delta
-3$, the edge $uv$ of $T$ is neither optimal nor suboptimal, a
contradiction with $T$ being nice.
\end{proof}

\begin{lemma}
\label{Lemma_niceExists}
Let $\Delta\geq4$ and let $n\geq\Delta+1$.
There exists a nice tree $T$ on $n$ vertices if and only if
$n$ is nice, i.e. if $n=k\Delta+1$.
Moreover, $T$ contains precisely $k$ vertices of degree $\Delta$.
\end{lemma}

\begin{proof}
Denote by $n_{i}$ the number of vertices of degree $i$ in $T$.
If $T$ is
nice, then Lemma \ref{Lemma_badSuboptimal} implies that every edge of $T$ is
either optimal or good suboptimal, so every vertex of $T$ has its degree in
the set $\{1,2,\Delta\}$, i.e. $n=n_{\Delta}+n_{2}+n_{1}$.
Also, every vertex
of degree $2$ is adjacent to two vertices of degree $\Delta$, hence
$n_{2}=n_{\Delta}-1$.
The Handshaking lemma implies $\Delta n_{\Delta}%
+2n_{2}+n_{1}=2(n-1)$.
We obtained a system of three linear equations with a unique solution $n_{\Delta}=\frac{1}{\Delta}\left(
n-1\right)$.
This implies $n=\Delta n_{\Delta}+1$, so $n$ is nice.

On the other hand, let $n=k\Delta+1$ be nice.
We construct a nice tree on
$n$ vertices by first forming a path $P=v_{1}u_{2}v_{2}\ldots u_{k}v_{k}$, and
then attaching $\Delta-2$ leaves to each vertex $v_{i}$ for $1\leq i\leq k$.
Finally, we attach one more leaf to $v_{1}$ and another new leaf  $v_{k}$.
The constructed tree has $n=k\Delta+1$ vertices and it is nice.
\end{proof}

In a nice tree $m=m_{\Delta1}+m_{\Delta2}$, where $m_{\Delta
2}=2(k-1)$ and $m_{\Delta1}=n-1-2(k-1)$.
So for a nice tree $T$
\begin{align*}
\sigma(T) &= 2(k-1)(\Delta-2)^{2}+(n-2k+1)(\Delta-1)^{2}\\
&= \Delta^3k-2\Delta^2k-3\Delta k+6k+4\Delta-6.
\end{align*}
For given $\Delta$, the value of $\sigma(T)$ depends only on $n$.
So every nice tree on $n$ vertices has the same value of $\sigma$-irregularity.
As Theorem~{\ref{Tm_general_exact}} below establishes, in the case of nice $n$ a tree is indeed maximal if and only if it is nice.
It remains to consider the number of vertices $n$ which is not nice.
For that purpose, we define three families of
graphs $\mathcal{T}_{n,\Delta}^{\prime}$, $\mathcal{T}_{n,\Delta}^{\prime\prime}$ and $\mathcal{T}_{n,\Delta}^{\prime\prime\prime}$, where for the sake of completeness we point out that in the case of nice $n$ all nice trees are in
$\mathcal{T}_{n,\Delta}^{\prime\prime}$.

\begin{definition}
\label{Def_Tprime}
Let $n=k\Delta+1+j,$ where $k\geq\Delta-1$ and $1\leq
j\leq\Delta-1$, be a non-nice integer. The family $\mathcal{T}_{n,\Delta
}^{\prime}$ consists of all trees on $n$ vertices with maximum degree $\Delta$
which have the following properties:

\begin{itemize}
\item[$(P_{1}^{\prime})$] all internal leaves of $T$ have degree
$\Delta$;

\item[$(P_{2}^{\prime})$] $m_{22}\leq2$ and $m_{\Delta\Delta}=0$;

\item[$(P_{3}^{\prime})$] $\sum_{i=3}^{\Delta-1}n_{i}\leq1$;

\item[$(P_{4}^{\prime})$] $m_{22}=0$ if and only if $\sum_{i=3}^{\Delta
-1}n_{i}=n_{a}=1$ for some $a\in\{3,\ldots,\Delta-1\}$;

\item[$(P_{5}^{\prime})$] if $m_{22}=0$, then $m_{a2}=a$, where $a$ is the
degree from $(P_{4}^{\prime})$.
\end{itemize}
\end{definition}

\begin{proposition}
\label{Prop_T1}
Let $n=k\Delta+1+j,$ where $k\geq\Delta-1$ and $1\leq
j\leq\Delta-1$, be a non-nice integer.
Then $\mathcal{T}_{n,\Delta}^{\prime}$ is not empty and all trees in $\mathcal{T}_{n,\Delta}^{\prime}$ have the same value of $\sigma$-irregularity.
\end{proposition}

\begin{proof}
First we prove that all trees in $\mathcal{T}_{n,\Delta
}^{\prime}$ have the same value of $\sigma$-irregularity. It is sufficient to
prove that for a tree $T\in\mathcal{T}_{n,\Delta}^{\prime}$, the value of
$\sigma(T)$ depends only on $n$. Notice that $n=k\Delta+1+j$ and $j\leq
\Delta-1$ imply that $j$ and $k$ depend only on $n$.
So if we show that $\sigma(T)$ depends only on $\Delta$, $j$ and $k$, that also means that $\sigma(T)$ depends only on $n$.

Assume first that $m_{22}\not =0$.
We wish to show that $n_{i}$ depends only on $n$ for every $i$. First, properties $(P_{3}^{\prime})$ and $(P_{4}^{\prime})$ imply that $T$ contains only vertices of degrees $1$, $2$ and $\Delta$, i.e. $n_{i}=0$ for $3\leq i\leq\Delta-1$.
Hence, $n_{\Delta}+n_{2}+n_{1}=n$.
By the Handshaking lemma we have $\Delta n_{\Delta}+2n_{2}+n_{1}=2(n-1)$.
Next, by suppressing all vertices of degree two in
$T$, we obtain a tree $T^{\prime}$ which has the same number of vertices of
degree $\Delta$ and $1$ as $T$.
Hence, $n_{1}=\Delta n_{\Delta}-2(n_{\Delta
}-1)$. Finally, property $(P_{1}^{\prime})$ implies $m_{12}=0$, so the
neighbors of a $2$-vertex can only be $\Delta$-vertices and $2$-vertices.
Since $m_{22}\leq2$ and $m_{\Delta\Delta}=0$ according to the property $(P_{2}^{\prime})$, $2$-vertices of $T$ are obtained by subdividing edges of $T^{\prime}$ with both end-vertices of degree $\Delta$ by precisely one $2$-vertex, and then
additional two $2$-vertices can be added.
Hence, $n_{2}=n_{\Delta}-1+t$ for $1\leq t\leq2$. 
(The case $t=0$ is excluded since $n$ is not nice.)
We obtain a system of four linear equations with the unique
solution
\begin{align*}
n_{1}  &  =\frac{1}{\Delta}\left(n\Delta-2n-\Delta t+2t+\Delta+2\right)  ,\\
n_{2}  &  =\frac{1}{\Delta}\left(n+\Delta t-t-\Delta-1\right)  ,\\
n_{\Delta}  &  =\frac{1}{\Delta}\left(n-t-1\right)  .
\end{align*}
Since $n_{\Delta}$ must be an integer and $n=k\Delta+1+j$, where $j\leq
\Delta-1$, we conclude that $t=j$ and $n_{\Delta}=k$. Notice that $t\leq2$
implies that trees with $m_{22}\not =0$ can occur in $\mathcal{T}_{n,\Delta
}^{\prime}$ only for $j\leq2$. So, we have established that $n_{1}$, $n_{2}$
and $n_{\Delta}$ depend only on $n$, and not on $T$.

Notice that $\sigma(T)=m_{\Delta1}(\Delta-1)^{2}+m_{\Delta2}(\Delta-2)^{2}$,
so next we wish to show that $m_{\Delta1}$ and $m_{\Delta2}$ also depend only
on $n$.
We have $m_{\Delta1}=n_{1}$, so $m_{\Delta1}$ depends only on $n$.
Since $m_{12}=0$ and $m_{\Delta\Delta}=0$, we have $m_{\Delta2}=n-1-m_{\Delta 1}-m_{22}$.
Notice that $m_{22}=t$, so $t=j$ implies that $m_{22}$ is determined by $n$.
We conclude that $m_{\Delta2}$ depends only on $n$. Hence, $\sigma(T)$
depends only on $n$ and we are done.

Assume next that $m_{22}=0$. Then by $(P_{4}^{\prime})$ we have $\sum
_{i=3}^{\Delta-1}n_{i}=n_{a}=1$ for some $a\in\{3,\ldots,\Delta-1\}$, so
$n=n_{\Delta}+n_{2}+n_{1}+1$. By Handshaking lemma we have $\Delta n_{\Delta
}+a+2n_{2}+n_{1}=2(n-1)$. Since $m_{21}=0$, each $2$-vertex of $T$ is neighbor
of $\Delta$-vertices and $a$-vertices. Properties $(P_{2}^{\prime})$ and
$(P_{5}^{\prime})$ imply that precisely $a$ $2$-vertices of $T$ are neighbors
of one $a$-vertex and one $\Delta$-vertex, and all other $2$-vertices are
neighbors of two $\Delta$-vertices.
Hence, $n_{2}=n_{\Delta}$.
We obtain a system of three linear equations with the unique solution%
\[
n_{1}=\frac{1}{\Delta}\left(n\Delta-2n-\Delta+2a+2\right)  ,\text{
\ \ }n_{2}=n_{\Delta}=\frac{1}{\Delta}\left(  n-a-1\right)  .
\]
Since $n_{\Delta}$ must be an integer and $n=k\Delta+1+j$, where $j\leq
\Delta-1$, we conclude that $a=j$. Notice that $a\geq3$ implies that trees
with $m_{22}=0$ can occur in $\mathcal{T}_{n,\Delta}^{\prime}$ only for
$j\geq3$. So, we have that $n_{\Delta}$, $n_{2}$ and $n_{1}$ depend only on
$n$.

Notice that $(P_{5}^{\prime})$ implies $m_{\Delta a}=m_{a1}=0$. Hence,
$\sigma(T)=m_{\Delta1}(\Delta-1)^{2}+m_{\Delta2}(\Delta-2)^{2}+m_{a
2}(a-2)^{2}$. Since $m_{\Delta1}=n_{1}$, we have that $m_{\Delta1}$ depends
only on $n$.
Further, $m_{a2}=a=j$ implies that $m_{a2}$ also depends only on $n$.
Finally, $m_{\Delta2}=n-1-m_{\Delta1}-m_{a2}$, so $m_{\Delta2}$ as well depends
only on $n$. We conclude that $\sigma(T)$ depends only on $n$.

It remains to show that the family $\mathcal{T}_{n,\Delta}^{\prime}$ is not
empty.
For $j\leq2$, let $T$ be a tree obtained from a path $z_{1}%
z_{2}\cdots z_{k}$ by first subdividing each of its edges with precisely one
new vertex, then subdividing the edge incident to $z_{1}$ with $j$ further new
vertices, and finally attaching leaves to each $z_{i}$ so that the degree of
$z_{i}$ in $T$ equals $\Delta$.
For $j\geq3$, let $T$ be a tree obtained
from the star with the center $w$ and leaves $z_{1},\ldots,z_{j}$, extended by
the path $z_{j}z_{j+1}\cdots z_{k}$, by subdividing each of its $k$ edges with
precisely one new vertex and attaching leaves to each $z_{i}$ so that the
degree of $z_{i}$ in $T$ equals $\Delta$. Notice that in both cases
$k\geq\Delta-1\geq j$ implies that the construction is well defined, where for
$k=j$ the extending path is trivial.
The number of vertices of the constructed tree equals $k\Delta+1+j=n$, and the tree has all the properties $(P_{1}^{\prime})$--$(P_{5}^{\prime})$.
\end{proof}

\begin{definition}
\label{Def_Tdprime}
Let $n=k\Delta+1+j,$ where $k\geq\Delta-1$ and $0\leq j\leq\Delta-1$.
The family $\mathcal{T}_{n,\Delta}^{\prime\prime}$ consists of all trees on $n$ vertices with maximum degree $\Delta$ which have the following properties:

\begin{itemize}
\item[$(P_{1}^{\prime\prime})$] all internal leaves of $T$ have degree
$\Delta$;

\item[$(P_{2}^{\prime\prime})$] all vertices of $T$ have their degree in the
set $\{1,2,\Delta\}$;

\item[$(P_{3}^{\prime\prime})$] $m_{22}=0$ and $m_{\Delta\Delta}\leq\Delta-1$.
\end{itemize}
\end{definition}

Observe that in Definition~{\ref{Def_Tdprime}} we admit $n$ to be nice.
So let $n=k\Delta+1$.
If $n_{\Delta}\le k-1$ then $n_1\le\Delta+(k-2)(\Delta-2)$ and so $n_2\ge\Delta+k-2$ contradicting $m_{22}=0$.
On the other hand if $n_{\Delta}\ge k+1$ then $n_1\ge\Delta+k(\Delta-2)$ and so $n_2\le k-\Delta$ contradicting $m_{\Delta\Delta}\le\Delta-1$.
So $n_{\Delta}=k$, $n_1=\Delta+(k-1)(\Delta-2)$ and $n_2=k-1$ implying $m_{\Delta\Delta}=0$.
So if $n$ is nice, $\mathcal{T}_{n,\Delta}^{\prime\prime}$ consists of nice trees.

\begin{proposition}
\label{Prop_T2}Let $n=k\Delta+1+j,$ where $k\geq\Delta-1$ and $1\leq
j\leq\Delta-1$, be a non-nice integer.
Then $\mathcal{T}_{n,\Delta}^{\prime\prime}$ is not empty and all trees in $\mathcal{T}_{n,\Delta}^{\prime\prime}$ have the same value of $\sigma$-irregularity.
\end{proposition}

\begin{proof}
Property $(P_{2}^{\prime\prime})$ implies $n_{\Delta}+n_{2}+n_{1}=n$.
Similarly as before, by suppressing vertices of degree two in $T$ we
obtain a tree $T^{\prime}$ which has the same number of $1$-vertices and $\Delta
$-vertices as $T$, so $n_{1}=\Delta n_{\Delta}-2(n_{\Delta}-1)$. The property
$(P_{1}^{\prime\prime})$ implies $m_{12}=0$, while $(P_{3}^{\prime\prime})$
implies $m_{22}=0$, from which we conclude that each $2$-vertex of $T$ is a
neighbor only to $\Delta$-vertices. This yields $m_{2\Delta}=2n_{2}$. Notice
that only $m_{\Delta\Delta}$, $m_{2\Delta}$ and $m_{1\Delta}$ contribute to
$\left\vert E(T)\right\vert $, so from $m_{1\Delta}=n_{1}$ we obtain
$m_{\Delta\Delta}+m_{2\Delta}+n_{1}=n-1$.

Hence, we obtain a system of four linear equations in terms of $n_{1},n_{2},n_{\Delta},m_{2\Delta}$ and $m_{\Delta\Delta}$ with the solution
\begin{align*}
n_{1}  &  =\frac{1}{\Delta}\left(n\Delta-2n+\Delta m_{\Delta\Delta}+\Delta-2m_{\Delta\Delta}+2\right)  ,\\
n_{2}  &  =\frac{1}{\Delta}\left(n-\Delta m_{\Delta\Delta}-\Delta+m_{\Delta\Delta}-1\right)  ,\\
n_{\Delta}  &  =\frac{1}{\Delta}\left(  n+m_{\Delta\Delta}-1\right)  ,\\
m_{2\Delta}  &  =\frac{1}{\Delta}\left(2n-2\Delta m_{\Delta\Delta}-2\Delta+2m_{\Delta\Delta}-2\right)  .
\end{align*}
Since $n_{\Delta}$ must be an integer and $n=k\Delta+1+j$, we obtain
$m_{\Delta\Delta}=\Delta-j$.
Since the value of $j$ is determined by $n$, this means $m_{\Delta\Delta}$ is also determined by $n$.
Consequently, $n_{1}$, $n_{2}$, $n_{\Delta}$ and $m_{2\Delta}$ are determined by $n$.
Since $\sigma(T)=m_{2\Delta}(\Delta-2)^{2}+m_{1\Delta}(\Delta-1)^{2}$ and
$m_{1\Delta}=n_{1}$, we conclude that $\sigma(T)$ is determined by $n$ and it
does not depend on $T$, so all trees of $\mathcal{T}_{n,\Delta}^{\prime\prime
}$ have the same value of $\sigma$-irregularity as claimed.
Observe that substituting $\Delta-j$ for $m_{\Delta\Delta}$ gives $n_{\Delta}=k+1$, $n_2=k+j-\Delta$ and $n_1=k\Delta+\Delta-2k$.

It remains to show that the family $\mathcal{T}_{n,\Delta}^{\prime\prime}$ is
not empty.
Let $T$ be a tree obtained from a path $z_{1}z_{2}\cdots z_{k+1}$
by subdividing each of the edges $z_{i}z_{i+1}$ for $1\leq i\leq k-(\Delta-j)$
with precisely one new vertex, and then attaching leaves to each $z_{i}$ so
that the degree of $z_{i}$ in $T$ equals $\Delta$. Notice that $k\geq
\Delta-1\geq\Delta-j$ implies that the construction is well defined, where for
$k=\Delta-j$ no edge of the path is subdivided.
The number of vertices of the constructed tree equals $k\Delta+1+j=n$, and it has all the properties $(P_{1}^{\prime\prime})$--$(P_{3}^{\prime\prime})$,
with $m_{\Delta\Delta}=\Delta-j$.
\end{proof}

In addition to families $\mathcal{T}_{n,\Delta}^{\prime}$ and
$\mathcal{T}_{n,\Delta}^{\prime\prime}$, one more family of trees is
needed for our characterization.
A vertex of $T$ whose degree belongs to the
set $\{3,\ldots,\Delta-1\}$ will be called a \emph{medium} vertex.
Trees of the third family contain precisely one medium vertex, which is of degree $j+1$, and whose neighbors are $j$ vertices of degree $2$ and one vertex of degree $\Delta$.
More precisely, we define the following family of trees.

\begin{definition}
\label{Def_Tstar}Let $n=k\Delta+1+j,$ where $k\geq\Delta-1$ and $2\leq
j\leq\Delta-2$.
The family $\mathcal{T}_{n,\Delta}^{\prime\prime\prime}$
consists of all trees on $n$ vertices with maximum degree $\Delta$ which have the following properties:

\begin{itemize}
\item[$(P_{1}^{\prime\prime\prime})$] all internal leaves of $T$ have
degree $\Delta$;

\item[$(P_{2}^{\prime\prime\prime})$] all vertices of $T$ have their degree in
the set $\{1,2,\Delta\}$, except precisely one medium vertex $w$ which is of
degree $j+1$;

\item[$(P_{3}^{\prime\prime\prime})$] $m_{22}=0$;

\item[$(P_{4}^{\prime\prime\prime})$] the medium vertex $w$ has precisely $j$
neighbors of degree $2$ and one neighbor of degree $\Delta$;

\item[$(P_{5}^{\prime\prime\prime})$] $m_{\Delta\Delta}=0$.
\end{itemize}

\noindent Notice that for $j\not \in \{2,3,\ldots,\Delta-2\}$ the degree $j+1$
of the vertex $w$ would not be medium, so in that case we consider the family
$\mathcal{T}_{n,\Delta}^{\prime\prime\prime}$ to be empty.
\end{definition}

\begin{proposition}
\label{Prop_Tstar}
Let $n=k\Delta+1+j,$ where $k\geq\Delta-1$ and $2\leq j\leq\Delta-2$.
Then $\mathcal{T}_{n,\Delta}^{\prime\prime\prime}$ is not empty and all trees in $\mathcal{T}_{n,\Delta}^{\prime\prime\prime}$ have the same value of $\sigma$-irregularity.
\end{proposition}

\begin{proof}
Property $(P_{2}^{\prime\prime\prime})$ implies $n_{\Delta}+1+n_{2}+n_{1}=n$.
By suppressing all vertices of degree two in $T$ we obtain a tree $T^{\prime}$ with the same number of vertices of degrees $\Delta$ and $1$ as $T$, in which the degree of $w$ is still $j+1$.
This yields $n_{1}=\Delta n_{\Delta}+(j+1)-2n_{\Delta}=(\Delta-2)n_{\Delta}+j+1$.
By $(P_{1}^{\prime\prime\prime})$ and $(P_{5}^{\prime\prime\prime})$ each $2$-vertex of $T$ is a neighbor only of $\Delta$-vertices and of $w$.
By property $(P_{4}^{\prime\prime\prime})$, precisely $j$ of the
$2$-vertices of $T$ are neighbors of $w$, and the second neighbor of each of
them is a $\Delta$-vertex, while every other $2$-vertex is a neighbor of two
$\Delta$-vertices.
So $m_{2\Delta}=2n_{2}-j$.
Further, in $T^{\prime}$ the vertex $w$ is adjacent to $j+1$ vertices of degree $\Delta$ and every leaf is adjacent to a vertex of degree $\Delta$.
Since $T^{\prime}$ has $n_{\Delta}+n_{1}+1$ vertices and therefore $n_{\Delta}%
+n_{1}$ edges, it contains precisely $n_{\Delta}-j-1$ edges with both
end-vertices of degree $\Delta$.
By $(P_{5}^{\prime\prime\prime})$, each such edge of $T^{\prime}$ is
subdivided in $T$ by precisely one $2$-vertex, so it contributes
two edges to $m_{2\Delta}$.
Since precisely $j$ edges of $T^{\prime}$
incident to $w$ are subdivided, each of them contributing one edge to
$m_{2\Delta}$, we have $\frac{1}{2}\left(m_{2\Delta}-j\right)=n_{\Delta}-j-1$.

Hence, we obtain a system of four linear equations in terms of
$n_{1},n_{2},n_{\Delta}$ and $m_{2\Delta}$ with the solution
\begin{align*}
n_{1}  &  =\frac{1}{\Delta}\left(n\Delta-2n+2j+2\right)  ,\\
n_{2}  &  =\frac{1}{\Delta}\left(n-\Delta-j-1\right)  ,\\
n_{\Delta}  &  =\frac{1}{\Delta}\left(n-j-1\right)  ,\\
m_{2\Delta}  &  =\frac{1}{\Delta}\left(2n-\Delta j-2\Delta-2j-2\right)  .
\end{align*}
Since $n_{\Delta}$ is integer and $n=k\delta+1+j$, we have $n_{\Delta}=k$, $n_{2}=k-1$, $n_{1}=(\Delta-2)k+j+1$ and $m_{2\Delta
}=2(k-1)-j$.

Notice that $\sigma(T)=m_{1\Delta}(\Delta-1)^{2}+m_{2\Delta}(\Delta
-2)^{2}+m_{2,j+1}(j-1)^{2}+m_{\Delta,j+1}(\Delta-j-1)^{2}$. Property
$(P_{4}^{\prime\prime\prime})$ implies $m_{2,j+1}=j$ and $m_{\Delta,j+1}=1$,
and since the medium vertex $w$ has no leaf neighbor, we have $m_{1\Delta
}=n_{1}$. We conclude that $\sigma(T)$ is determined by $n$ and it does not
depend on $T$, so all trees of $\mathcal{T}_{n,\Delta}^{\prime\prime\prime}$
have the same value of $\sigma$-irregularity as claimed.

It remains to show that the family $\mathcal{T}_{n,\Delta}^{\prime\prime
\prime}$ is not empty.
Let $T$ be a tree obtained from a star with center $w$ and leaves $z_{1},\ldots,z_{j+1}$, extended by the path $z_{j+1}z_{j+2}\cdots z_{k}$.
We subdivide each of its $k$ edges except the edge $wz_{j+1}$ with precisely one new vertex, and then we attach leaves to each $z_{i}$ so that the degree of $z_{i}$ in $T$ equals $\Delta$.
Notice that
$j\leq\Delta-2$ and $k\geq\Delta-1$ imply $k\geq j+1$, so the construction is
well defined, where for $k=j+1$ the extending path is trivial.
The number of vertices of the constructed tree equals $k\Delta+1+j=n$, and it has all the properties $(P_{1}^{\prime\prime\prime})$--$(P_{5}^{\prime\prime\prime})$.
\end{proof}

For an illustration of trees from the families $\mathcal{T}_{n,\Delta
}^{\prime}$, $\mathcal{T}_{n,\Delta}^{\prime\prime}$ and $\mathcal{T}%
_{n,\Delta}^{\prime\prime\prime}$ see Figure \ref{Fig_families}.

\begin{figure}[h]
\begin{center}
\includegraphics[width=11.53cm]{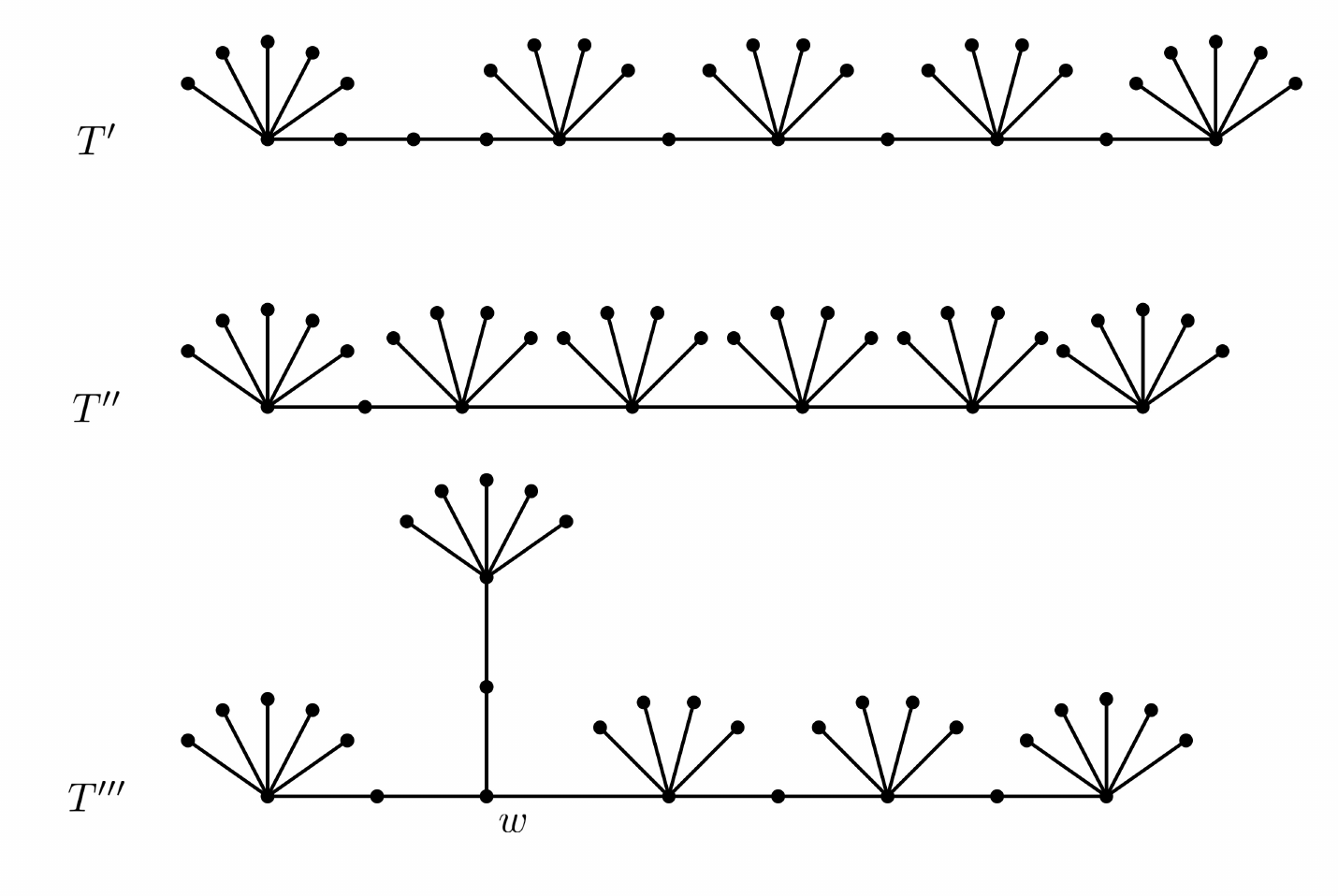}
\end{center}
\caption{The figure shows a tree $T^{\prime}\in\mathcal{T}_{33,6}^{\prime}$, a
tree $T^{\prime\prime}\in\mathcal{T}_{33,6}^{\prime\prime}$ and a tree
$T^{\prime\prime\prime}\in\mathcal{T}_{33,6}^{\prime\prime\prime}$, i.e.
$n=k\Delta+1+j$ for $\Delta=6$, $k=5$ and $j=2$. According to Proposition
\ref{Prop_greatest}, in this case it holds that $\sigma(T^{\prime\prime\prime
})=\sigma(T^{\prime\prime})>\sigma(T^{\prime})$.}%
\label{Fig_families}%
\end{figure}

Propositions \ref{Prop_T1}, \ref{Prop_T2} and \ref{Prop_Tstar} immediately
yield the following corollary.

\begin{corollary}
\label{Cor_comparison}Let $n=k\Delta+1+j,$ where $k\geq\Delta-1$ and $1\leq
j\leq\Delta-1$, be a non-nice integer.
Let $T^{\prime}\in\mathcal{T}_{n,\Delta}^{\prime}$, $T^{\prime\prime}\in\mathcal{T}_{n,\Delta}^{\prime\prime}$, and in the case $2\leq j\leq\Delta-2$ let also
$T^{\prime\prime\prime}\in\mathcal{T}_{n,\Delta}^{\prime\prime\prime}$.
Then
\begin{align*}
\sigma(T^{\prime})  &  =\left\{
\begin{array}
[c]{cc}%
\Delta^{3}k-2\Delta^{2}k-3\Delta k+6k+4\Delta-6 & \text{for }j\leq 2;\\
\Delta^{3}k-2\Delta^{2}k-3\Delta k+6k+2\Delta j+j^{3}-4j^{2}+j &
\text{for }j\geq 3;
\end{array}
\right. \\
\sigma(T^{\prime\prime})  &  =\Delta^{3}k-2\Delta^{2}k-3\Delta k+6k-\Delta^{3}+2\Delta^{2}j+6\Delta^{2}-8\Delta j-7\Delta+8j,\\
\sigma(T^{\prime\prime\prime})  &  =\Delta^{3}k-2\Delta^{2}k-3\Delta
k+6k+4\Delta+j^{3}-j^{2}-6.
\end{align*}
\end{corollary}

\begin{proof}
The edge multiplicities are determined in the proofs of Propositions~{\ref{Prop_T1}},
{\ref{Prop_T2}} and {\ref{Prop_Tstar}}.
For $T^{\prime}$ with $j\leq2$ we have $m_{\Delta 1}=(\Delta-2)k+2$, $m_{\Delta2}=2(k-1)$ and $\sigma(T^{\prime})=m_{\Delta 1}(\Delta-1)^2+m_{\Delta 2}(\Delta-2)^2$, which gives the result.

For $T^{\prime}$ with $j\geq3$ we have $m_{\Delta1}=(\Delta-2)k+j$, $m_{\Delta2}=2k-j$, $m_{a2}=a=j$ and $\sigma(T^{\prime})=m_{\Delta 1}(\Delta-1)^2+m_{\Delta 2}(\Delta-2)^2+m_{a 2}(a-2)^2$.

For $T^{\prime\prime}$ we have $m_{1\Delta}=(\Delta-2)k+\Delta$, $m_{2\Delta
}=2(k+j-\Delta)$ and $\sigma(T^{\prime})=m_{\Delta 1}(\Delta-1)^2+m_{\Delta 2}(\Delta-2)^2$.

Finally, for $T^{\prime\prime\prime}$ we have $m_{1\Delta
}=(\Delta-2)k+j+1$, $m_{2\Delta}=2(k-1)-j$, $m_{2,j+1}=j$, $m_{\Delta
,j+1}=1$ and $\sigma(T^{\prime})=m_{\Delta 1}(\Delta-1)^2+m_{2\Delta}(\Delta-2)^2+m_{2,j+1}(j-1)^2+m_{\Delta,j+1}(\Delta-j-1)^2$. Expanding the expression yields the resilt.
\end{proof}

Observe that the formula for $T'$ and $j\le 2$ is the same as the formula for a nice tree.
The reason is that the trees $T'$ for $j=1$ and $j=2$ are obtained from nice trees by subdividing $2\Delta$-edges.
Hence, if we admit $j=0$, $m_{22}=0$ and $\sum_{i=3}^{\Delta-1}=0$ in Definition~{\ref{Def_Tprime}}, then for $j=0$ we get nice trees.
Consequently, if $n$ is nice, then the maximal trees on $n$ and $n+1$ vertices achieve the same
value of $\sigma$-irregularity.
Indeed, for $n+1$ we have $j=1$, so the family $\mathcal{T}_{n,\Delta}^{\prime\prime\prime}$ is
empty and $\sigma^{\prime}>\sigma^{\prime\prime}$ by part $(i)$ of Proposition~{\ref{Prop_greatest}} below, since $1<r(\Delta)$ for every $\Delta\geq6$.
Let us stress that this does not extend to $n+2$: for $j=2$ part $(iv)$ of Proposition~{\ref{Prop_greatest}} yields $\sigma^{\prime\prime\prime}=\sigma^{\prime}+4$, and
$\sigma^{\prime\prime\prime}\geq\max\{\sigma^{\prime},\sigma^{\prime\prime}\}$, so the maximum
$\sigma$-irregularity on $n+2$ vertices exceeds the one on $n$ vertices by exactly $4$ for every $\Delta\geq6$.

In the rest of the paper, we denote by $\sigma^{\prime}$, $\sigma
^{\prime\prime}$ and $\sigma^{\prime\prime\prime}$ the common value of
$\sigma$-irregularity of all trees from the family $\mathcal{T}_{n,\Delta
}^{\prime}$, $\mathcal{T}_{n,\Delta}^{\prime\prime}$ and $\mathcal{T}%
_{n,\Delta}^{\prime\prime\prime}$, respectively. Notice that these values are
well defined, since Propositions \ref{Prop_T1}, \ref{Prop_T2} and
\ref{Prop_Tstar} imply that all trees of each of the three families have the
same value of $\sigma$-irregularity, while Corollary \ref{Cor_comparison}
gives these values explicitly.

\begin{proposition}
\label{Prop_greatest}
Let $n=k\Delta+1+j,$ where $k\geq\Delta-1$ and $1\leq
j\leq\Delta-1$, be a non-nice integer, and let $r(\Delta)=\frac{1}{2}\left(
\sqrt{5\Delta^{2}-32\Delta+44}-\Delta+4\right)  $ denote the positive root of
$q(j)=j^{2}+(\Delta-4)j-(\Delta^{2}-6\Delta+7)$. Then the following holds:

\begin{itemize}
\item[\textrm{(i)}]
$\sigma^{\prime}>\sigma^{\prime\prime}$ if $j<r(\Delta)$;

\item[\textrm{(ii)}]
$\sigma^{\prime}=\sigma^{\prime\prime}$ if $j=r(\Delta)$;

\item[\textrm{(iii)}]
$\sigma^{\prime\prime}>\sigma^{\prime}$ if $j>r(\Delta)$;

\item[\textrm{(iv)}] for $2\leq j\leq\Delta-2$ it holds that $\sigma
^{\prime\prime\prime}<\max\{\sigma^{\prime},\sigma^{\prime\prime}\}$, except
the following four cases:
\[
\begin{array}
[c]{ll}
j=2\text{ and }\Delta\geq7: & \sigma^{\prime\prime\prime}=\sigma^{\prime
}+4>\max\{\sigma^{\prime},\sigma^{\prime\prime}\},\\
(\Delta,j)=(8,3): & \sigma^{\prime\prime\prime}=\sigma^{\prime}+2>\max
\{\sigma^{\prime},\sigma^{\prime\prime}\},\\
(\Delta,j)=(6,2): & \sigma^{\prime\prime\prime}=\sigma^{\prime\prime}%
>\sigma^{\prime},\\
(\Delta,j)=(9,3): & \sigma^{\prime\prime\prime}=\sigma^{\prime}>\sigma
^{\prime\prime}.
\end{array}
\]

\end{itemize}
\end{proposition}

\begin{proof}
Notice first that in all the formulas of Corollary \ref{Cor_comparison} the
summands which contain $k$ are the same, namely $\Delta^{3}k-2\Delta
^{2}k-3\Delta k+6k$. Hence, in the difference of any two of the values
$\sigma^{\prime}$, $\sigma^{\prime\prime}$ and $\sigma^{\prime\prime\prime}$
these summands cancel out, so the differences depend only on $\Delta$ and $j$.

$(i)-(iii)$ A direct computation yields
\[
\sigma^{\prime}-\sigma^{\prime\prime}=\left\{
\begin{array}
[c]{ll}%
(\Delta-2)(4j-4\Delta-2j\Delta+\Delta^{2}+3) & \text{for }j\leq2,\\[2pt]%
(j-\Delta)\,q(j) & \text{for }j\geq3.
\end{array}
\right.
\]
Notice that the discriminant $5\Delta^{2}-32\Delta+44$ of $q$ is positive for
$\Delta\geq6$ and that $q(0)<0$, so $q$ has a unique positive root $r(\Delta)$.

Assume first that $j\geq3$. Then $j\leq\Delta-1<\Delta$ implies $j-\Delta<0$,
so the sign of $\sigma^{\prime}-\sigma^{\prime\prime}$ is opposite to the sign
of $q(j)$. Hence, $\sigma^{\prime}-\sigma^{\prime\prime}>0$ holds if and only
if $q(j)<0$, i.e. if and only if $j<r(\Delta)$.
The cases $(ii)$ and $(iii)$ are direct consequences of this observation.

Assume next that $j\leq2$.
In this case $\sigma^{\prime}-\sigma^{\prime\prime
}>0$ holds if and only if $j<L(\Delta)$, where $L(\Delta)=\frac{(\Delta
-1)(\Delta-3)}{2(\Delta-2)}$. For $j=1,$ both inequalities which constitute
$(i)$, namely $j<L(\Delta)$ and $j<r(\Delta),$ hold for every $\Delta\geq6.$
Equality cannot hold in either of them, so the claim $(i)$ holds.
For $j=2$ both inequalities hold if $\Delta\geq7$.
For $\Delta=6$ and $j=2$ we have $j>\frac{15}{8}=L(6)$ and $j>2\sqrt{2}-1=r(6)$, so in this case $(iii)$ applies.

$(iv)$ In order to compare $\sigma^{\prime\prime\prime}$ with $\sigma^{\prime
}$, denote $A=\sigma^{\prime\prime\prime}-\sigma^{\prime}$.
We get
\[
A=\left\{
\begin{array}
[c]{ll}%
j^{2}(j-1) & \text{for }j\leq2,\\[2pt]%
-2\Delta(j-2)+3j^{2}-j-6 & \text{for }j\geq3.
\end{array}
\right.
\]
Similarly, comparing $\sigma^{\prime\prime\prime}$ with $\sigma^{\prime\prime
}$ we obtain $\sigma^{\prime\prime\prime}-\sigma^{\prime\prime}=(\Delta
-j-1)P,$ where
\[
P=\Delta^{2}-(j+5)\Delta-j^{2}+2j+6.
\]
Since $j\leq\Delta-2$ implies $\Delta-j-1\geq1$, the sign of $\sigma
^{\prime\prime\prime}-\sigma^{\prime\prime}$ coincides with the sign of $P$.
Notice further that $A<0$ implies $\sigma^{\prime\prime\prime}<\sigma^{\prime
}\leq\max\{\sigma^{\prime},\sigma^{\prime\prime}\}$, and similarly $P<0$
implies $\sigma^{\prime\prime\prime}<\sigma^{\prime\prime}\leq\max
\{\sigma^{\prime},\sigma^{\prime\prime}\}$. Hence, it is sufficient to show
that for every pair $(\Delta,j)$ outside the four exceptional cases at least
one of the values $A$ and $P$ is negative.

Assume first that $j=2$.
Then $A=4$, so $\sigma^{\prime\prime\prime}$ always exceeds $\sigma^{\prime}$ by $4$, which solves the case $\Delta\ge 7$.
For $\Delta=6$ we have $P=0$, and so $\sigma^{\prime\prime\prime}=\sigma^{\prime\prime}$.
This gives the first and the third exceptional cases.

Assume now that $j\geq3$.
Notice that $A$ is linear and $P$ is a quadratic function of $\Delta$.
And for $3\leq j\leq\Delta-2$, the function $A$ is decreasing, while $P$ is increasing for $\Delta\geq6$.

For $j=3$ we have $A=18-2\Delta$ and $P=\Delta^{2}-8\Delta+3$. If $\Delta
\in\{6,7\}$, then $P<0$. If $\Delta=8$, then $A=2$ and $P=3$, so
$\sigma^{\prime\prime\prime}=\sigma^{\prime}+2$ exceeds both $\sigma^{\prime}$
and $\sigma^{\prime\prime}$, which is the second exceptional case. If
$\Delta=9$, then $A=0$ and $P=12$, so $\sigma^{\prime\prime\prime}%
=\sigma^{\prime}>\sigma^{\prime\prime}$, which is the fourth exceptional case.
Finally, if $\Delta\geq10$, then $A<0$.

Consider next $j\in\{4,5,6\}$. In these three cases $A$ equals $38-4\Delta$,
$64-6\Delta$ and $96-8\Delta$, respectively, so $A\geq0$ holds if and only if
$\Delta\leq9$, $\Delta\leq10$ and $\Delta\leq12$, respectively. For all such
$\Delta$ the value of $P$ is negative, since $P$ is increasing and the value
of $P$ on the right end-point of the $\Delta$ interval is negative for the
corresponding $j.$ Hence, for $j\in\{4,5,6\}$ at least one of the values $A$
and $P$ is negative for every $\Delta$.

Finally, assume that $j\geq7.$ If $A<0,$ then the claim holds, so let us
assume that $A\geq0$. The inequality $A\geq0$ implies $\Delta\leq\left(
3j^{2}-j-6\right)  /(2j-4),$ and for this $\Delta$ the value of $P$ is
negative. Since $P$ is increasing for $\Delta\geq6,$ it follows that $P$ is
negative for all $\Delta$ in this case, which concludes the proof of claim
$(iv)$.
\end{proof}

Let us mention that the equality $\sigma^{\prime}=\sigma^{\prime\prime}$,
i.e. $j=r(\Delta)$, can occur only for $j\geq3$ and requires $5\Delta
^{2}-32\Delta+44$ to be a perfect square.
The smallest instances are $(\Delta,j)\in\{(14,7),(77,46),(509,313)\}$.

Now we unify the families $\mathcal{T}_{n,\Delta}^{\prime}$,
$\mathcal{T}_{n,\Delta}^{\prime\prime}$ and $\mathcal{T}_{n,\Delta}%
^{\prime\prime\prime}$ into one family of extremal trees according to
Proposition \ref{Prop_greatest}.
Recall that for a nice $n$ we have
$\mathcal{T}_{n,\Delta}=\mathcal{T}_{n,\Delta}^{\prime\prime}$.
In this case $\mathcal{T}_{n,\Delta}$ consists of nice trees.

For $n=k\Delta+1+j$, where $k\geq\Delta-1$ and
$1\leq j\leq\Delta-1$, we define the family of trees%
\[
\mathcal{T}_{n,\Delta}=\left\{
\begin{array}
[c]{ll}%
\mathcal{T}_{n,\Delta}^{\prime\prime\prime} & \text{if }j=2\text{ and }%
\Delta\geq7,\text{ or }(\Delta,j)=(8,3),\\
\mathcal{T}_{n,\Delta}^{\prime\prime}\cup\mathcal{T}_{n,\Delta}^{\prime
\prime\prime} & \text{if }(\Delta,j)=(6,2),\\
\mathcal{T}_{n,\Delta}^{\prime}\cup\mathcal{T}_{n,\Delta}^{\prime\prime\prime}
& \text{if }(\Delta,j)=(9,3),\\
\mathcal{T}_{n,\Delta}^{\prime} & \text{if }j<r(\Delta),\text{ in all other
cases},\\
\mathcal{T}_{n,\Delta}^{\prime}\cup\mathcal{T}_{n,\Delta}^{\prime\prime} &
\text{if }j=r(\Delta),\text{ in all other cases},\\
\mathcal{T}_{n,\Delta}^{\prime\prime} & \text{if }j>r(\Delta),\text{ in all
other cases}.
\end{array}
\right.
\]
Notice that, according to Proposition \ref{Prop_greatest}, the family
$\mathcal{T}_{n,\Delta}$ consists precisely of the trees from those families
among $\mathcal{T}_{n,\Delta}^{\prime}$, $\mathcal{T}_{n,\Delta}^{\prime
\prime}$ and $\mathcal{T}_{n,\Delta}^{\prime\prime\prime}$ whose common value
of $\sigma$-irregularity is the greatest.

By analogy with the characterization of maximal trees for $\Delta=5$ obtained
in \cite{ksvsd-sipmt-2026}, one could expect that for $\Delta\geq6$ a tree on
$n=k\Delta+1+j$ vertices with the maximum degree $\Delta$ is maximal if and
only if it belongs to the family obtained by unifying only $\mathcal{T}%
_{n,\Delta}^{\prime}$ and $\mathcal{T}_{n,\Delta}^{\prime\prime}$ according to
the comparison of $\sigma^{\prime}$ and $\sigma^{\prime\prime}$, i.e. without
the four exceptional cases of Proposition \ref{Prop_greatest}. Theorem
\ref{Tm_general_exact} below shows that such a characterization holds only
after the four exceptional cases are taken into account. Notice that in the
first two exceptional cases the family expected in this way is not maximal at
all, in the last two cases it is maximal but it does not contain all maximal
trees, while for every other pair $(\Delta,j)$ it remains unchanged.

Now we can state the main result of this paper.

\begin{theorem}
\label{Tm_general_exact}
Let $\Delta\geq6$, and let $n=k\Delta+1+j$, where $k\geq
\Delta-1$ and $0\leq j\leq\Delta-1$. Then the maximum $\sigma$-irregularity
over all trees on $n$ vertices with the maximum degree $\Delta$ equals
\[
\max\{\sigma^{\prime},\sigma^{\prime\prime},\sigma^{\prime\prime\prime}\},
\]
where the term $\sigma^{\prime\prime\prime}$ is present only when $2\leq
j\leq\Delta-2$. Moreover, a tree $T$ on $n$ vertices with the maximum degree
$\Delta$ is maximal if and only if $T\in\mathcal{T}_{n,\Delta}$.
\end{theorem}

\begin{proof}
The proof of Theorem \ref{Tm_general_exact} requires several lemmas on the
structure of maximal trees.
To make the paper easier to read, we state the theorem here, and we give these
lemmas and the proof in the next section.
\end{proof}

\section{The proof of the main theorem}

\label{sec:proof}

In this section we prove Theorem \ref{Tm_general_exact}. We first recall
several structural properties of maximal trees which hold for every
$\Delta\geq3$ and which were established in \cite{ksvsd-sipmt-2026}, namely
Lemmas \ref{Prop_2.1} and \ref{Prop_basic}. These results are stated
here without proofs.
After that, we establish three further properties of
maximal trees which we need in the case $\Delta\geq6$.

\begin{lemma}
\label{Prop_2.1}
Let $T$ be a maximal tree with maximum degree $\Delta$, and let $P=ux\cdots yv$ be a path in $T$.
If $d(u)>d(v)$, then $d(x)\leq d(y)$. Also, if $d(x)>d(y)$, then $d(u)\leq d(v)$.
\end{lemma}

\begin{lemma}
\label{Prop_basic}Let $T$ be a maximal tree on $n$
vertices with maximum degree $\Delta\geq3$. Then the following holds:

\begin{itemize}
\item[\textrm{(i)}] if $n\geq\Delta+4$, then $T$ contains at least two vertices whose degree is at least~$3$;

\item[\textrm{(ii)}] if $n\geq\Delta+4$, then $m_{21}=0$;

\item[\textrm{(iii)}] if $n\geq7$ for $\Delta=3$ and $n\geq2\Delta$ for
$\Delta\geq4$, then all internal leaves of $T$ have degree $\Delta$;

\item[\textrm{(iv)}] if $n\geq\Delta+4$, then $m_{22}\leq2$.
\end{itemize}
\end{lemma}

\bigskip

Now we prove some additional properties of maximal trees.

\begin{lemma}
\label{Prop_mdd}
Let $T$ be a maximal tree with maximum degree $\Delta \geq6$.
Then $m_{\Delta\Delta}\leq\Delta-1$.
\end{lemma}

\begin{proof}
Assume to the contrary that $m_{\Delta\Delta}\geq\Delta$.
Then $T$ contains two adjacent vertices of degree $\Delta$, so $n\geq2\Delta$.
Let $u$ be an internal leaf of $T$.
Part $(iii)$ of Lemma \ref{Prop_basic} implies that $u$ is a vertex of degree $\Delta$.
Since $m_{\Delta\Delta}\geq \Delta$, there exist at least $\Delta-1$ edges in $T$ with both end-vertices of degree $\Delta$ which are not incident to $u$.
Denote these edges by $x_{i}y_{i}$ for $i=1,\ldots,\Delta-1$.
Also, since $u$ is an internal leaf, it has $\Delta-1$ neighbors which are leaves, denote them by $u_{i}$ for $i=1,\ldots,\Delta-1$.
Let $v$ be the neighbor of $u$ which is not a leaf.
Denote $E^{-}=\{x_{i}y_{i},u_{i}u:i=1,\ldots,\Delta-1\}$ and $E^{+}%
=\{x_{i}u_{i},u_{i}y_{i}:i=1,\ldots,\Delta-1\}$.
Consider the tree $T^{\prime}=T-E^{-}+E^{+}$, and notice that edges $x_{i}y_{i}$ contribute to
$\sigma(T^{\prime})-\sigma(T)$ by zero, edges $u_{i}u$ by $-(\Delta-1)^{2}$,
and edges $x_{i}u_{i}$ and $u_{i}y_{i}$ by $(\Delta-2)^{2}$.
Hence, we have
\begin{align*}
\sigma(T^{\prime})-\sigma(T)  &  =-(\Delta-1)(\Delta-1)^{2}+2(\Delta
-1)(\Delta-2)^{2}+(1-d(v))^{2}-(\Delta-d(v))^{2}\\
&  =\left(  \Delta-1\right)  \left((\Delta-6)(\Delta-1)+2d(v)\right)  .
\end{align*}
Since $d(v)\geq2$, it follows that for $\Delta\geq6$ we have $\sigma
(T^{\prime})-\sigma(T)>0$, a contradiction.
\end{proof}

Let us mention that for $\Delta=5$ the sharper bound $m_{\Delta\Delta}\leq3$
is obtained in \cite{ksvsd-sipmt-2026} by a similar argument.

\begin{lemma}
\label{Prop_m22NonZero}
Let $T$ be a maximal tree with maximum degree $\Delta$.
If $m_{22}\not =0$, then the following holds:

\begin{itemize}
\item[\textrm{(i)}] every neighbor of $u\in V(T)$ with $d(u)\geq3$ is either a
leaf or a vertex of degree $2$;

\item[\textrm{(ii)}] $\sum_{i=3}^{\Delta-1}n_{i}=0$;

\item[\textrm{(iii)}] $m_{\Delta\Delta}=0$.
\end{itemize}
\end{lemma}

\begin{proof}
$(i)$ Assume to the contrary, that there exists a vertex $v$ in $T$ such that
$uv\in E(T)$ and $d(v)\geq3$. Also, let $xy\in E(T)$ be an edge with
$d(x)=d(y)=2$. Since $m_{22}\not =0$, such an edge must exist in $T$. We may
assume that vertices $u,v,x,y$ are denoted so that the path of $T$ which
connects $u$ and $y$ contains vertices $v$ and $x$. Since $d(u)>d(y)$,
Lemma \ref{Prop_2.1} implies $d(v)\leq d(x)$, a contradiction.

$(ii)$ Assume to the contrary that $T$ contains a vertex $x$ with $d(x)=i$ for
some $3\leq i\leq\Delta-1$.
Since $m_{22}\not =0$, there is an edge $ab\in E(T)$ with $d(a)=d(b)=2$.
Removing $ab$ splits $T$ into two components.
Since $d(x)\geq3$ we have $x\not \in \{a,b\}$, so we may label the end-vertices of
$ab$ so that $x$ lies in the component containing $a$.
In particular $b$ is not adjacent to $x$ in $T$.
Consider the tree
\[
T^{\prime}=T-ab+bx.
\]
Since $x$ and $b$ lie in different components of $T-ab$, the graph $T^{\prime}$
is again a tree; and since $i+1\leq\Delta$, its maximum degree is still
$\Delta$.

Notice that this transformation changes only the degree of $x$, which
increases from $i$ to $i+1$, and the degree of $a$, which decreases from $2$
to $1$, while the degree of $b$ remains $2$.
Hence, the only edges whose contribution to $\sigma$ changes are the deleted edge $ab$, the added edge $bx$, the $i$ edges incident to $x$, and the edge $aa^{\prime}$ which connects $a$ to its neighbor $a^{\prime}\not =b$.
By part $(i)$, every neighbor $c$ of
$x$ satisfies $d(c)\leq2$. Notice that the edge $ab$ contributes zero to the
difference $\sigma(T^{\prime})-\sigma(T)$, the edge $bx$ contributes
$(2-(i+1))^{2}=(i-1)^{2}$, and for every neighbor $c\not =a$ of $x$ the edge
$xc$ contributes $((i+1)-d(c))^{2}-(i-d(c))^{2}=2i-2d(c)+1\geq2i-3$, since
$d(c)\leq2$.

It remains to consider the edge $aa^{\prime}$.
If $a^{\prime}\not =x$, i.e. if $a$ is not a neighbor of $x$, then this edge is distinct from all edges incident to $x$ and it contributes $(d(a^{\prime})-1)^{2}-(d(a^{\prime})-2)^{2}=2d(a^{\prime})-3\ge -1$, since $d(a^{\prime})\ge 1$.
So we have
\[
\sigma(T^{\prime})-\sigma(T)\geq(i-1)^{2}+i(2i-3)-1=i(3i-5)>0.
\]

If $a^{\prime}=x$, i.e. if $a$ is a neighbor of $x$, then the edge
$aa^{\prime}=xa$ accounts for the degree change of both its end-vertices and
contributes $((i+1)-1)^{2}-(i-2)^{2}=4i-4$, while each of the remaining $i-1$
edges incident to $x$ contributes at least $2i-3$.
So we have
\[
\sigma(T^{\prime})-\sigma(T)\geq(i-1)^{2}+(4i-4)+(i-1)(2i-3)=3i(i-1)>0.
\]

In both cases we obtain $\sigma(T^{\prime})>\sigma(T)$, a contradiction with
$T$ being maximal.
Therefore, $T$ contains no vertex of degree $i$ for
$3\leq i\leq\Delta-1$, i.e. $\sum_{i=3}^{\Delta-1}n_{i}=0$.

$(iii)$ Since $m_{22}\not =0$, $T$ contains a subpath $uvwz$, where
$d(v)=d(w)=2$.
Assume to the contrary, that $m_{\Delta\Delta}\not =0$. Let
$xy\in E(T)$ be an edge with $d(x)=d(y)=\Delta$. Let $T^{\prime}%
=T-vw-wz+vz-xy+xw+wy$.
Then only the edges $xw$ and $wy$ contribute to
$\sigma(T^{\prime})-\sigma(T)$, each with $(\Delta-2)^{2}$, since the
contribution of edges $vw$ and $xv$ is $0$.
So we have
\[
\sigma(T^{\prime})-\sigma(T)=2(\Delta-2)^{2}>0,
\]
a contradiction with $T$ being maximal.
\end{proof}

\begin{lemma}
\label{Prop_m22zero}
Let $n=k\Delta+1+j,$ where $k\geq\Delta-1$ and $0\leq
j\leq\Delta-1$, and let $T$ be a maximal tree on $n$ vertices with the maximum
degree $\Delta$. If $j\geq3$, then $m_{22}=0$.
\end{lemma}

\begin{proof}
Assume to the contrary that $m_{22}>0$.
Let $uvwz$ be a subpath of $T$ such that $d(v)=d(w)=2$.
By Lemma~{\ref{Prop_m22NonZero}} $(iii)$, we have $m_{\Delta\Delta}=0$.
Further, part $(ii)$ of Lemma \ref{Prop_m22NonZero} implies that all vertices of $T$ have their
degree in the set $\{1,2,\Delta\}$. Also, part $(ii)$ of Lemma
\ref{Prop_basic} implies $m_{21}=0$, so every leaf of $T$ is adjacent to a
$\Delta$-vertex and both neighbors of every $2$-vertex of $T$ are of the
degree $2$ or $\Delta$.
Hence, by suppressing all vertices of degree two in $T$ we obtain a tree $T^{\prime}$ with the same number of vertices of degrees $\Delta$ and $1$ as $T$.
Since $T^{\prime}$ has $n_{\Delta}+n_{1}$
vertices and therefore $n_{\Delta}+n_{1}-1$ edges, of which $n_{1}$ are
incident to a leaf, it contains precisely $n_{\Delta}-1$ edges with both
end-vertices of degree $\Delta$.
Each such edge of $T^{\prime}$ corresponds to a path in $T$ whose internal vertices are $2$-vertices, where the number of internal vertices on each of these paths is at least one.
Since a path with $i$ internal vertices contains $i-1$
edges with both end-vertices of degree $2$, we have $n_{2}=(n_{\Delta
}-1)+m_{22}$.

On the other hand, from $n_{1}+n_{2}+n_{\Delta}=n$ and the
Handshaking lemma $n_{1}+2n_{2}+\Delta n_{\Delta}=2(n-1)$ it follows that
$n_{2}+(\Delta-1)n_{\Delta}=n-2=k\Delta+j-1$. Substituting $n_{2}%
=m_{22}+n_{\Delta}-1$ into this equality yields $m_{22}=\Delta(k-n_{\Delta
})+j$, so $m_{22}-j$ is divisible by $\Delta$.
Since part $(iv)$ of Lemma \ref{Prop_basic} implies $1\leq m_{22}\leq2$ and since $1\leq j\leq\Delta-1$, we have $m_{22}=j$.
Hence, $j\leq2$, a contradiction with $j\geq3$.
\end{proof}

Let us now show that maximal trees indeed belong to the three families
introduced in the previous section. Throughout the rest of the section we
assume $\Delta\geq6$ and $n=k\Delta+1+j$, where $k\geq\Delta-1$ and $0\leq
j\leq\Delta-1$. Notice that this implies $n\geq2\Delta$, so all the structural
properties established above hold for maximal trees on $n$ vertices. For the
sake of brevity, let us denote
\[
\mu=\max\{\sigma^{\prime},\sigma^{\prime\prime},\sigma^{\prime\prime\prime
}\},
\]
where the term $\sigma^{\prime\prime\prime}$ is omitted when
$j\not \in \{2,\ldots,\Delta-2\}$, and for $j=0$ the value of $\mu$ equals the value of $\sigma$-index of nice trees, see the expression below Lemma~{\ref{Lemma_niceExists}}.
Also, a vertex of degree $1$ or $2$ will be called a \emph{light} vertex.

Recall that by Lemma \ref{Prop_m22NonZero}, a maximal tree with
$m_{22}\not =0$ contains no medium vertex. In the following two lemmas we
therefore consider maximal trees with $m_{22}=0$. Recall also that, since
$n\geq\Delta+4$, part $(ii)$ of Lemma \ref{Prop_basic} yields $m_{12}%
=0$. We show first that a maximal tree with $m_{22}=0$ cannot contain two
medium vertices, and then we bound the value of $\sigma(T)$ in the case of
precisely one medium vertex.

For a vertex $w$ of $T$, we denote $S_{w}=\sum_{yw\in E(T)}d(y)$.
I.e., $S_w$ is the sum of degrees of all neighbors of $w$.

\begin{lemma}
\label{Lemma_exchange}
Let $T$ be a maximal tree with maximum degree $\Delta\geq6$.
If $m_{22}=0$, then $T$ contains at most one medium vertex.
\end{lemma}

\begin{proof}
Assume to the contrary, that $T$ contains at least two medium vertices. We say
that a medium vertex of degree $a$ is \emph{$\Delta$-heavy} if at least
$a-1$ of its neighbors are of degree $\Delta$.
We distinguish two cases.

\medskip

\noindent\textbf{Case 1.}
\emph{$T$ contains a $\Delta$-heavy medium vertex}.
Let $w$ be a $\Delta$-heavy medium vertex of degree $a$. 
Let $x$ be a neighbor of $w$ of degree $\Delta$.
Since $w$ has at least one more neighbor of degree
$\Delta$, the component of $T-wx$ containing $w$ contains a leaf $u$ of $T$ at
distance at least two from $w$. Denote by $v$ the neighbor of $u$ and notice
that $d(v)\leq\Delta$. Consider the tree $T^{\prime}=T-wx+xu$. Since $u$ and
$x$ lie in different components of $T-wx$, the graph $T^{\prime}$ is again a
tree. From $T$ to $T^{\prime}$, the contribution to $\sigma$ changes only on
the deleted edge $wx$, on the added edge $xu$, on the $a-1$ edges connecting
$w$ with its remaining neighbors, and on the edge $uv$.
This yields
\begin{align*}
\sigma(T^{\prime})-\sigma(T)  &  =-(\Delta-a)^{2}+(\Delta-2)^{2}+\sum_{yw\in
E(T),\,y\neq x}\left(  (d(y)-(a-1))^{2}-(d(y)-a)^{2}\right) \\
&  \quad+(d(v)-2)^{2}-(d(v)-1)^{2}.
\end{align*}
Each summand of the sum in the above formula equals $2(d(y)-a)+1$, which is an
increasing function of $d(y)$. Since $w$ is $\Delta$-heavy, at most one
neighbor of $w$ is not of degree $\Delta$, and the degree of that neighbor
is at least $1$. Also, it holds that $d(v)\leq\Delta$, and so $(d(v)-2)^2-(d(v)-1)^2=3-2d(v)\ge3-2\Delta$.
Plugging all these
values in the above expression, we obtain
\begin{align*}
\sigma(T^{\prime})-\sigma(T)  &  \ge (\Delta-2)^{2}-(\Delta-a)^{2}
+(a-2)(2\Delta-2a+1)+(2-2a+1)+(3-2\Delta)\\
& = -3a^{2}+\left(  4\Delta+3\right)  a+\left(
8-10\Delta\right)  =f(a).
\end{align*}
Since $f$ is a quadratic function of $a$ with the negative leading
coefficient, from $f(3)=2\Delta-10>0$ and $f(\Delta-1)=\Delta^{2}-5\Delta+2>0$
for $\Delta\geq6$ we conclude that $f(a)>0$ for every $3\leq a\leq\Delta-1$.
Hence, $\sigma(T^{\prime})>\sigma(T)$, a contradiction with $T$ being maximal.

\medskip

\noindent\textbf{Case 2.}
\emph{No medium vertex of $T$ is $\Delta$-heavy.}
In this case every medium vertex has at least two neighbors which are light or medium, and there are at least two medium vertices.
Let $u$ and $v$ be two medium vertices at the biggest distance in $T$.
Since neither $u$ nor $v$ are $\Delta$-heavy, there are edges $xu$ and $vy$ such that both $x$ and $y$ are medium or light and none of them is on the $u-v$ path in $T$.
By the choice of $u$ and $v$, the vertices $x$ and $y$ cannot be medium, so they must be light.
Denote by $a$ ($b$) the degree of $u$ ($v$).

First consider the tree $T_{1}=T-ux+vx$.
Since $v$ and $x$ lie in different components of $T-ux$, the graph $T_{1}$ is
again a tree.
The contribution to $\sigma$ changes only on the deleted edge $ux$, on the added edge $vx$, on the remaining $a-1$ edges incident to $u$, and on the edges incident to $v$.
This yields
\begin{align*}
\sigma(T_{1})-\sigma(T) &  =(d(x)-(b{+}1))^{2}-(d(x)-a)^{2}+\sum_{zu\in E(T),z\neq
x}\left(  (d(z)-(a{-}1))^{2}-(d(z)-a)^{2}\right)  \\
&  \quad+\sum_{zv\in E(T)}\left(  (d(z)-(b{+}1))^{2}-(d(z)-b)^{2}\right)
+\varepsilon\\
&  =3(a+b)(b-a+1)-2d(x)(b-a+2)-2(S_{v}-S_{u})+\varepsilon,
\end{align*}
where $\varepsilon=2$ if $u$ and $v$ are adjacent and $\varepsilon=0$
otherwise.

Now consider the tree $T_{2}=T-vy+uy$.
Analogously as in the case of $T_1$, we get
\begin{align*}
\sigma(T_{2})-\sigma(T) &  =(d(y)-(a{+}1))^{2}-(d(y)-b)^{2}+\sum_{zv\in E(T),z\neq
y}\left(  (d(z)-(b{-}1))^{2}-(d(z)-b)^{2}\right)  \\
&  \quad+\sum_{zu\in E(T)}\left(  (d(z)-(a{+}1))^{2}-(d(z)-a)^{2}\right)
+\varepsilon\\
&  =3(a+b)(a-b+1)-2d(y)(a-b+2)-2(S_{u}-S_{v})+\varepsilon.
\end{align*}

Summing the two differences, the terms with $S_{u}$ and $S_{v}$ cancel out, so
we obtain
\[
\big(\sigma(T_{1})-\sigma(T)\big)+\big(\sigma(T_{2})-\sigma
(T)\big)=6(a+b)-2d(x)(b-a+2)-2d(y)(a-b+2)+2\varepsilon.
\] 
Since $\varepsilon\geq0$ and $d(x),d(y)\in \{1,2\}$, the right-hand side of the above equality
is at least
\begin{align*}
a\big(6+2(d(x)-d(y))\big)+b\big(6-2(d(x)-d(y))\big)-4\big(d(x)+d(y)\big) & \ge \\
4\big(a+b-d(x)-d(y)\big) & >0,
\end{align*}
since $u$ and $v$ are medium vertices while $x$ and $y$ are light.
\end{proof}

\begin{lemma}
\label{Lemma_onemedium}
Let $T$ be a maximal tree with maximum degree $\Delta\geq6$.
If $m_{22}=0$ and $T$ contains precisely one medium vertex,
then $\sigma(T)\leq\max\{\sigma^{\prime},\sigma^{\prime\prime},\sigma
^{\prime\prime\prime}\}$, with equality only if
$T\in\mathcal{T}_{n,\Delta}^{\prime}
\cup\mathcal{T}_{n,\Delta}^{\prime\prime\prime}$.
\end{lemma}

\begin{proof}
Denote by $w$ the medium vertex of $T$ and let $a=d(w)\in\{3,\ldots
,\Delta-1\}$.
Further, denote by $p$, $q$ and $r$ the number of neighbors of
$w$ which are leaves, $2$-vertices and vertices of degree $\Delta$,
respectively, so $p+q+r=a$.
Since $m_{22}=m_{12}=0$, every $2$-vertex of $T$ subdivides an edge between two
vertices of degree at least $3$.
By suppressing all vertices of degree two similarly as in the proof of Proposition \ref{Prop_Tstar}, we obtain
\[
n_{1}=(\Delta-2)n_{\Delta}+a,\qquad n_{2}=n-1-(\Delta-1)n_{\Delta}-a,\qquad m_{\Delta\Delta}=n_{\Delta}-n_{2}-r,
\]
where the last identity follows since the suppressed tree contains precisely
$n_{\Delta}-q-r$ edges having both vertices of degree $\Delta$, and each of the $n_{2}-q$ vertices of degree two which are not adjacent to $w$ subdivides
a distinct one of them.
Since every edge of $T$ connects vertices of degrees $(1,\Delta)$,
$(1,a)$, $(2,\Delta)$, $(2,a)$, $(\Delta,\Delta)$ or $(a,\Delta)$, we have
\begin{equation}
\label{For_starmed}
\sigma(T)=(n_{1}-p)(\Delta-1)^{2}+p(a-1)^{2}+(2n_{2}-q)(\Delta-2)^{2}%
+q(a-2)^{2}+r(\Delta-a)^{2}.
\end{equation}
Notice that replacing one $(\Delta,2)$-edge and one $(a,1)$-edge by one $(\Delta,1)$-edge and one $(a,2)$-edge, i.e. decreasing $p$ by one and increasing $q$ by one, changes the value of (\ref{For_starmed}) by
$(\Delta-1)^{2}+(a-2)^{2}-(a-1)^{2}-(\Delta-2)^{2}=2(\Delta-a)>0$.
Similarly, inserting a 2-vertex into a $(\Delta,\Delta)$-edge instead of $(\Delta,a)$-edge, i.e. decreasing $q$ by one and increasing $r$ by one, changes the value of (\ref{For_starmed}) by $2(\Delta-2)^{2}+(\Delta-a)^{2}-(\Delta-2)^2-(a-2)^{2}=2(\Delta-a)(\Delta-2)>0$.
Hence, when $T$ is maximal, we may assume $p=0$ and take $r$ as
large as possible, subject to $m_{\Delta\Delta}=n_{\Delta}-n_{2}-r\geq0$.

Let us now show that $n_{\Delta}\in\{k,k+1\}$.
Denote $t=n_{\Delta}-n_{2}$.
Substituting the above expression for $n_{2}$ and using $n=k\Delta+1+j$, we obtain
$t=\Delta(n_{\Delta}-k)+(a-j)$.
Since the suppressed tree has $n_{\Delta}$ edges and $m_{12}=m_{22}=0$, we have $t\geq0$.
Since $a\leq\Delta-1$ implies $a-j<\Delta$, this yields $\Delta
(n_{\Delta}-k)>-\Delta$, i.e. $n_{\Delta}\geq k$.
On the other hand, since $m_{\Delta\Delta}\le\Delta-1$ by Lemma~{\ref{Prop_mdd}}, we have $t-a\le n_{\Delta}-n_2-r=m_{\Delta\Delta}\le\Delta-1$, i.e. $\Delta(n_{\Delta}-k)\leq\Delta-1+j\leq2\Delta-2$, so $n_{\Delta}\leq k+1$.
We distinguish two cases.

\medskip

\noindent\textbf{Case 1.}
\emph{$n_{\Delta}=k+1$.}
Notice that in this case the above inequality $\Delta(n_{\Delta}-k)\leq\Delta-1+j$ yields $j\geq1$, so $n$ is not nice.
Further, $n_{2}=k+j+1-\Delta-a$ and $t=\Delta+a-j>a$, so the greatest possible value of $r$ equals $a$, which yields $q=0$ and $m_{\Delta\Delta}=\Delta-j$.
Substituting these values into (\ref{For_starmed}) we obtain
$$
\sigma(T)\le \big((\Delta-2)(k+1)+a\big)(\Delta-1)^2
+2(k+j+1-\Delta-a)(\Delta-2)^2+a(\Delta-a)^2.
$$
By Corollary~{\ref{Cor_comparison}}, for $3\leq a\leq\Delta-1$ we get
$$
\sigma(T)-\sigma^{\prime\prime}\le (a-1)(a-2)(a+3-2\Delta)<0.
$$
Hence, in this case we have $\sigma(T)<\sigma^{\prime\prime}\leq\mu$,
so the inequality is strict.

\medskip

\noindent\textbf{Case 2.}
\emph{$n_{\Delta}=k$.}
Then $n_{2}=k+j-a$ and $t=a-j$, so the
condition $t\geq0$ yields $a\geq j$. The greatest possible value of $r$ equals
$a-j$, which yields $q=j$ and $m_{\Delta\Delta}=0$. Substituting these values
into (\ref{For_starmed}) we obtain
\begin{align*}
\sigma(T)& \le
\big((\Delta-2)k+a\big)(\Delta-1)^2+(2k+2j-2a-j)(\Delta-2)^2\\
& \quad+j(a-2)^2+(a-j)(\Delta-a)^2\\
& = \Delta^{3}k-2\Delta^{2}k-3\Delta k+6k+8j-4\Delta j-2\Delta a^{2}%
+a^{3}+a(2\Delta j+6\Delta-4j-7).
\end{align*}
Denote by $g(a)$ the expression on the right-hand side of the above inequality.

Further, denote by
\[
\sigma^{\ast}=\Delta^{3}k-2\Delta^{2}k-3\Delta k+6k+4\Delta+j^{3}-j^{2}-6
\]
the expression which defines $\sigma^{\prime\prime\prime}$ in Corollary
\ref{Cor_comparison}.
We show that $\sigma^{\ast}\leq\mu$ for every
$0\leq j\leq\Delta-1$.
Namely, for $2\leq j\leq\Delta-2$ we have $\sigma^{\ast}=\sigma^{\prime\prime\prime}\leq\mu$.
For $j=\Delta-1$ a direct computation yields that $\sigma^{\ast}$ equals the value of $\sigma^{\prime\prime}$ from Corollary \ref{Cor_comparison}.
For $j\in\{0,1\}$ the summand $j^{3}-j^{2}$ vanishes, so $\sigma^{\ast}$ reduces to the formula for $\sigma^{\prime}$ with $j\leq2$, which equals $\sigma^{\prime}$ by Corollary
\ref{Cor_comparison} for $j=1$, and for $j=0$ as well, see the remark below Corollary~{\ref{Cor_comparison}}.

A direct computation yields that $g(j)$ equals the formula for $\sigma
^{\prime}$ with $j\ge 3$, while $g(j+1)=\sigma^{\ast}$.
Since $a\geq3$, the value $a=j$ is possible only for $j\geq3$, in which case $g(j)=\sigma^{\prime}\leq\mu$, and the value $a=j+1$ is possible only for
$2\le j\leq\Delta-2$, in which case $g(j+1)=\sigma^{\prime\prime\prime}\leq\mu$.
It remains to prove the strict inequality $g(a)<\mu$ for all the remaining values of $a$, i.e. for every integer $a$ with $\max\{3,j+2\}\leq a\leq\Delta-1$.
(Recall that $a\ge j$ since $t\ge 0$.)

Observe that $g(a)-\sigma^{\ast}=(a-j-1)R(a)$, where
\[
R(a)=a^{2}+(j+1-2\Delta)a+4\Delta+j^{2}-2j-6.
\]
Introducing $e=a-j-2\ge 0$ and $d=\Delta-1-a\geq0$, which measure the
distance of $a$ from the two extreme values of its range, i.e. substituting
$a=j+2+e$ and $\Delta=j+3+e+d$, we obtain
\[
R(a)=j^{2}-j(e+1)-e(e+1)-2d(j+e).
\]
If $R(a)<0$, then $g(a)-\sigma^{\ast}=(a-j-1)R(a)<0$,
so $\sigma(T)\le g(a)<\sigma^{\ast }\leq\mu$.
Assume therefore that $R(a)\geq0$.
In this case we must have
$d<j$, since for $d\geq j$ the inequality $j+e>0$ would imply
\[
R(a)\leq j^{2}-j(e+1)-e(e+1)-2j(j+e)=-j(j+3e+1)-e(e+1)<0,
\]
where the last strict inequality holds since $j+3e+1\ge 2$ for $j\geq1$, while for
$j=0$ we have $e=a-2>0$.
Hence $j>d\geq0$, which means that $n$ is not nice.
By Corollary~{\ref{Cor_comparison}}, we get
$g(a)-\sigma^{\prime\prime}=(a-\Delta)Q(a)$, where
\[
Q(a)=a^{2}-\Delta a-\Delta^{2}+2\Delta j+6\Delta-4j-7=R(a)+(e+d+2)(j-d).
\]
Now $R(a)\geq0$ and $j-d>0$ yield $Q(a)>0$, so from $a<\Delta$ we obtain
$\sigma(T)\le g(a)<\sigma^{\prime\prime}\leq\mu$.

\medskip

Summing up both cases $n_{\Delta}=k+1$ and $n_{\Delta}=k$, we conclude that
$\sigma(T)\leq\mu$.
Moreover, the equality is possible only for $n_{\Delta}=k$ and $a\in\{j,j+1\}$, where it also forces $p=0$ and the greatest possible value of $r$, since the two
replacements of neighbors of $w$ considered above strictly increase the value
of (\ref{For_starmed}).
This corresponds to $T\in\mathcal{T}_{n,\Delta}^{\prime}$ in the case $a=j$ and $r=0$, and to $T\in\mathcal{T}_{n,\Delta}^{\prime\prime\prime}$ in the case $a=j+1$ and $r=1$, where in the latter case $m_{\Delta\Delta}=0$ yields the property $(P_{5}^{\prime\prime\prime})$.
\end{proof}

We are now in a position to prove Theorem \ref{Tm_general_exact}.

\begin{proof}[Proof of Theorem \ref{Tm_general_exact}]
Let $T$ be a maximal tree. Since
$n\geq2\Delta$, parts $(ii)$ and $(iii)$ of Lemma \ref{Prop_basic} imply
that $m_{12}=0$ and that all internal leaves of $T$ are of degree $\Delta$.

Assume first that $m_{22}\not =0$.
Lemma \ref{Prop_m22NonZero} implies that $T$ contains no medium vertex, $m_{\Delta\Delta}=0$, and every neighbor of a vertex of degree at least $3$ is either a leaf or a vertex of degree $2$.
Together with $m_{22}\leq2$ from Lemma~{\ref{Prop_basic}}, this yields that $T$ has all the properties $(P_{1}^{\prime})$--$(P_{5}^{\prime})$.
The counting argument from the proof of Proposition~{\ref{Prop_T1}}, which uses only these properties, then yields $m_{22}=t$ for some $t\in\{1,2\}$ with $t\equiv j\pmod{\Delta}$, so $j=t\in\{1,2\}$ and, in particular, $n$ is not nice.
Hence, $T\in\mathcal{T}_{n,\Delta}^{\prime}$ and
Proposition \ref{Prop_T1} implies $\sigma(T)=\sigma^{\prime}$.

Now assume that $m_{22}=0$.
Lemma~{\ref{Lemma_exchange}} implies that $T$
contains at most one medium vertex.
If $T$ contains no medium vertex, then all
vertices of $T$ have their degree in the set $\{1,2,\Delta\}$.
So $T\in\mathcal{T}_{n,\Delta}^{\prime\prime}$.
Proposition~{\ref{Prop_T2}} and Corollary~{\ref{Def_Tdprime}} imply $\sigma(T)=\sigma^{\prime\prime}$ for $j\ge 1$, while for a nice $n$ the argument below Definition~{\ref{Def_Tdprime}} shows $m_{\Delta\Delta}=0$ and $\sigma(T)$ is evaluated below Lemma~{\ref{Lemma_niceExists}}.
If $T$ contains precisely one medium vertex, then Lemma~{\ref{Lemma_onemedium}} implies $\sigma(T)\leq\mu$, where the equality implies $T\in\mathcal{T}_{n,\Delta}^{\prime}\cup
\mathcal{T}_{n,\Delta}^{\prime\prime\prime}$.

In every case we have $\sigma(T)\leq\mu$. Further, for a non-nice $n$
Propositions \ref{Prop_T1} and \ref{Prop_T2} imply that the families
$\mathcal{T}_{n,\Delta}^{\prime}$ and $\mathcal{T}_{n,\Delta}^{\prime\prime}$
are not empty and that their trees attain the values $\sigma^{\prime}$ and
$\sigma^{\prime\prime}$ given in Corollary \ref{Cor_comparison}, while for
$2\leq j\leq\Delta-2$ Proposition \ref{Prop_Tstar} implies the same for the
family $\mathcal{T}_{n,\Delta}^{\prime\prime\prime}$ and the value
$\sigma^{\prime\prime\prime}$.
For a nice $n$, Lemma~{\ref{Lemma_niceExists}} provides a tree on $n$ vertices, and its $\sigma$-irregularity is counted below Lemma~{\ref{Lemma_niceExists}}.
(See also the remark below Corollary~{\ref{Cor_comparison}}.)
We conclude that the maximum $\sigma$-irregularity over all trees on $n$ vertices with the maximum degree $\Delta$ equals $\mu$.

Finally, since a maximal tree attains $\mu$, from the above analysis it
follows that a maximal tree with $m_{22}\not =0$ belongs to $\mathcal{T}%
_{n,\Delta}^{\prime}$, a maximal tree without medium vertices belongs to
$\mathcal{T}_{n,\Delta}^{\prime\prime}$, and a maximal tree with precisely one
medium vertex belongs to $\mathcal{T}_{n,\Delta}^{\prime}\cup\mathcal{T}%
_{n,\Delta}^{\prime\prime\prime}$. Hence, a maximal tree belongs to a family
among $\mathcal{T}_{n,\Delta}^{\prime}$, $\mathcal{T}_{n,\Delta}^{\prime
\prime}$, $\mathcal{T}_{n,\Delta}^{\prime\prime\prime}$ whose trees attain the
value $\mu$, and the union of those families is precisely $\mathcal{T}%
_{n,\Delta}$. We conclude that every maximal tree belongs to $\mathcal{T}%
_{n,\Delta}$. Conversely, every tree of $\mathcal{T}_{n,\Delta}$ attains the
value $\mu$, so it is maximal, and the proof is complete.
\end{proof}

\section{Concluding remarks}

In this paper we characterized trees which attain the maximum value of
$\sigma$-irregularity among all trees on $n$ vertices with maximum degree
$\Delta\geq6$.
Together with the results of \cite{kpvsd-sict-2024} for
chemical trees, i.e. for $\Delta\leq4$, and of \cite{ksvsd-sipmt-2026} for
$\Delta=5$, this completes the characterization of maximal trees with
prescribed maximum degree for all values of $\Delta$, provided that $n$ is
sufficiently large.
It turned out that, in contrast with the case $\Delta=5$, the
two natural candidate families $\mathcal{T}_{n,\Delta}^{\prime}$ and
$\mathcal{T}_{n,\Delta}^{\prime\prime}$ do not suffice for $\Delta\geq6$.
Namely, in four exceptional cases the third family $\mathcal{T}_{n,\Delta}^{\prime\prime\prime}$ appears:
the infinite family of pairs $(\Delta,j)$ with $j=2$ and $\Delta\geq7$, and the three
individual pairs $(\Delta,j)\in\{(6,2),(8,3),(9,3)\}$.

\bigskip

\begingroup\sloppy
\noindent\textbf{Acknowledgments.}~~The authors acknowledge the partial support
by Slovak research grants APVV 22-0005, APVV 23-0076, VEGA 1/0069/23 and
VEGA 1/0011/25, by ARIS projects J1-3002 and J1-70016, program P1-0383,
bilateral Slovenian-Croatian project BI-HR/25-27-004 and the annual work
program of Rudolfovo, by Project KK.01.1.1.02.0027 co-financed by the European
Regional Development Fund, by Croatian Ministry of Science, Education and
Youth through the bilateral Croatian-Slovenian project 2025--26, and by the
NextGeneration EU foundation via IP-UNIST-17 (GEORAZ).
\par\endgroup

\end{document}